\documentclass[12pt,reqno]{amsart}

\usepackage{amsmath,amsfonts,amsthm,amssymb,amscd, dsfont,tipa,mathtools,graphics,diagbox,listings,placeins,verbatim}
\usepackage{lmodern}
\usepackage{hyperref}
\usepackage{color}
\usepackage{breakurl}
\usepackage{comment}
\newcommand{\bburl}[1]{\textcolor{blue}{\url{#1}}}

\usepackage{color}
\numberwithin{equation}{section}

\newcommand\be{\begin{equation}}
\newcommand\ee{\end{equation}}
\newcommand\bea{\begin{eqnarray}}
\newcommand\eea{\end{eqnarray}}
\newcommand\bi{\begin{itemize}}
\newcommand\ei{\end{itemize}}
\newcommand\ben{\begin{enumerate}}
\newcommand\een{\end{enumerate}}

\newtheorem{thm}{Theorem}[section]
\newtheorem{conj}[thm]{Conjecture}
\newtheorem{cor}[thm]{Corollary}
\newtheorem{lem}[thm]{Lemma}
\newtheorem{prop}[thm]{Proposition}

\theoremstyle{definition}

\newtheorem{defi}[thm]{Definition}

\newtheorem{que}[thm]{Question}

\theoremstyle{remark}
\newtheorem{rek}[thm]{Remark}
\newtheorem{cla}[thm]{Claim}
\newtheorem{calc}[thm]{Calculation}

\newcommand{\su}{\sum}
\newcommand{\sui}[1]{\su_{i=#1}^\infty}

\newcommand{\pnhb}[1]{P_{#1}^\text{H}}
\newcommand{\pnhz}{\pnhb{n}}
\newcommand{\pnh}[1]{\pnhz(#1)}
\newcommand{\pnmb}[2][P]{#1_{#2}^\text{M}}
\newcommand{\pnmz}{\pnmb{n}}
\newcommand{\pnm}[1]{\pnmz(#1)}
\newcommand{\qnmb}[1]{\pnmb[Q]{#1}}

\newcommand{\tand}{{\rm\ and\ }}
\newcommand{\tor}{{\rm \ or\ }}
\newcommand{\ifff}{\Longleftrightarrow}
\newcommand{\lin}{\lim\limits_{n\to\infty}}
\newcommand{\novt}{\frac{n}{2}}
\newcommand{\oovt}{\frac{1}{2}}
\newcommand{\oovf}{\frac{1}{4}}
\newcommand{\tovf}{\frac{3}{4}}
\newcommand{\fovn}{\frac{4}{9}}
\newcommand{\tvfp}{\left(\frac{3}{4}\right)}
\newcommand{\cond}{\text{ \big | }}
\newcommand{\cht}[1][0,n-1\in]{\cond #1 S}
\newcommand{\misk}{|S-S|=2n-1-k}
\newcommand{\kitn}{$k\in 2\mathbb{N}$}
\newcommand{\eq}[1]{\begin{align}#1\end{align}}
\newcommand{\smeq}[1]{{\small\eq{#1}}}
\newcommand{\eqn}[1]{\begin{align*}#1\end{align*}}

\newcommand{\twocase}[5]{#1 \begin{cases} #2 & #3\\ #4
& #5 \end{cases}   }

\newcommand{\threecase}[7]{#1 \begin{cases} #2 &
#3\\ #4 & #5\\ #6 & #7 \end{cases}   }

\begin{document}

\title{Distribution of missing differences in diffsets}

\author{Scott Harvey-Arnold} \address[Scott Harvey-Arnold]{Carnegie Mellon University, Pittsburgh, PA 15213}\email{sharveyarnold@gmail.com}
\author{Steven J. Miller} \address[Steven J. Miller]{Carnegie Mellon University, Pittsburgh, PA 15213} \curraddr{Department of Mathematics and Statistics, Williams College, Williamstown, MA 01267}\email{sjm1@williams.edu}
\author{Fei Peng} \address[Fei Peng]{Carnegie Mellon University, Pittsburgh, PA 15213}\email{fpeng1@andrew.cmu.edu}

\date{\today}

\begin{abstract} Lazarev, Miller and O'Bryant \cite{LMO} investigated the distribution of $|S+S|$ for $S$ chosen uniformly at random from $\{0, 1, \dots, n-1\}$, and proved the existence of a divot at missing 7 sums (the probability of missing exactly 7 sums is less than missing 6 or missing 8 sums). We study related questions for $|S-S|$, and shows some divots from one end of the probability distribution, $P(|S-S|=k)$, as well as a peak at $k=4$ from the other end, $P(2n-1-|S-S|=k)$. A corollary of our results is an asymptotic bound for the number of complete rulers of length $n$.
\end{abstract}



\thanks{This work was partially supported by NSF grant DMS1561945, Carnegie Mellon University, and Williams College. We thank Joshua Siktar for the constructive comments, and Carnegie Mellon University for the research opportunity and the AFS computing platform.}

\maketitle


\section{Introduction}

\subsection{Background}
Let $S$ be a typical subset of \be [n]\ := \ \{0, 1, \dots, n-1\};\ee in other words, we choose $S$ uniformly at random, or equivalently each integer in $[n]$ is independently chosen to be in $S$ with probability $1/2$. Define \eq{S+S \ := \ \{x+y:x,y\in S\}\tand S-S:=\{x-y:x,y\in S\}.} We refer to these as the \textit{sumset} and the \textit{diffset} of $S$, and we denote the cardinality of a set $A$ by $|A|$.

The sizes of the sumset and the diffset have been compared extensively. As addition is commutative and subtraction is not, it was conjectured that as $n\to\infty$ almost all sets $S$ should be difference dominated: $|S-S| > |S+S|$. Thus while sum-dominant sets were known to exist, and constructions for infinite families were given, they were thought to be rare. This conjecture turns out to be false; Martin and O'Bryant \cite{MO} proved that for a small but positive proportion of all subsets of $[n]$, the sumset has a larger cardinality than the diffset. This result holds if instead of choosing each element with probability 1/2 we instead choose with a fixed probability $p > 0$; however, if $p$ is allowed to decay to zero with $n$ then Hegarty and Miller \cite{HM} proved almost all sets are difference dominated. For these and related results see \cite{AMMS, BELM, CLMS, CMMXZ, DKMMW, DKMMWW, He, HLM, ILMZ, Ma, MOS, MP, MS, MV, MXZ,  Na1, Na2, Ru1, Ru2, Ru3, Zh1, Zh2}.


The distribution of $|S+S|$ has also been studied. When $S$ is chosen uniformly at randomly from $[n]$, Lazarev, Miller and O'Bryant \cite{LMO} proved an unusual ``divot'' occurs in the limiting probability distribution of $|S+S|$ (the existence of the limiting distribution was shown by Zhao \cite{Zh2}). In particular, the limiting probability of missing 7 sums is less than that of missing 6 (or 8): {\footnotesize\eq{ \lin P(2n-1-|S+S|=7)<\lin P(2n-1-|S+S|=6)<\lin P(2n-1-|S+S|=8).}} Further, \cite{LMO} gave rigorous bounds for $\lin P(2n-1-|S+S|=k)$ for $0\le k<32$, which imply that there are no more divots until $k=27$. It is unknown whether there could be more divots later.  Figure 1 of their paper is reproduced here with permission as Figure \ref{figure:x_k Experimental graph}.

\begin{figure}[h]
\begin{center}
\includegraphics[scale=.80]{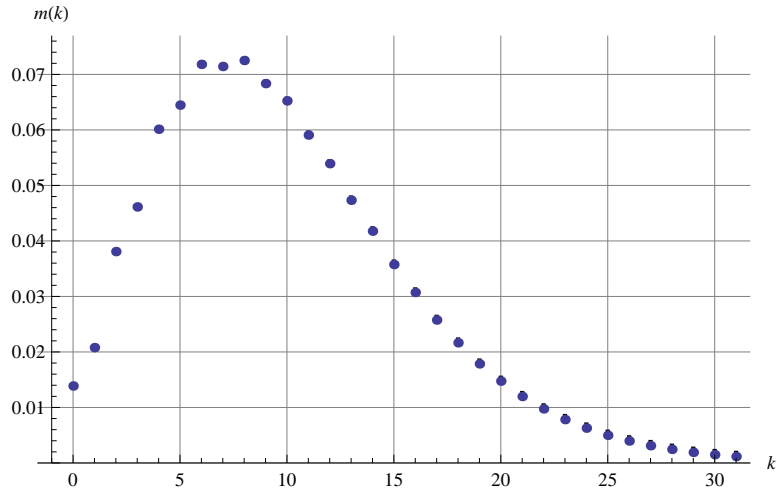}
\end{center}
\caption{Experimental values of $m(k)$, the probability $S+S$ is missing exactly $k$ sum, when each $m \in [n]$ is in $S$ with probability $1/2$. The vertical bars depict the values allowed by the most rigorous bounds in \cite{LMO}. In most cases, the allowed interval is smaller than the dot indicating the experimental value. The data comes from generating $2^{28}$ sets uniformly forced to contain 0 from $[0, 256)$.\label{figure:x_k Experimental graph}}
\end{figure}

However, the probability distribution of $|S-S|$, the size of the diffset, has not been extensively investigated. One reason for the success in $|S+S|$ and the lack of progress for $|S-S|$ is that the sumset is significantly easier to exhaustively investigate. For many sets, their properties can be determined by decomposing $S$ as $L \cup M \cup R$, where $L$ and $R$ are respectively the left and right fringe elements and $M$ is the middle; typically $L$ and $R$ are of bounded size independent of $n$, so most elements in $S$ are in $M$. As there are many ways to write a number as a sum or difference of elements, most elements in $[n]+[n]$ or $[n]-[n]$ are realized, especially since a typical $S$ has on the order of $n/2$ elements and thus generates on the order of $n^2/2$ pairs. The difference is for the fringe elements, where there are fewer representations and thus a greater chance of an element not being obtained.\footnote{An integer $m \le n$ can be written as $m+1$ sums of pairs of elements from $[n]$, and if $m$ is modest it is thus unlikely that none of these pairs have both elements in $S$; however, if $m$ is small then an element can have a significant probability of not occurring. For example, if $0 \in S$ but $1 \not\in S$ then $1 \not\in S+S$.} For sumsets the left and right fringes do not interact, with the left fringe $L + L$ and the right $R+R$; this is not the case for the diffset, where the fringes are $L-R$ and its negative $R-L$. As a result, to determine whether an extremal element is in $S+S$, only one fringe matters while for $S-S$, both ends must be considered. The computational complexity is hence \emph{squared}, which makes the diffset distribution significantly harder to exhaustively investigate.

Below we focus on the probability distribution of $|S-S|$.

\subsection{Distribution of $|S-S|$ when $n=35$}\

We display the probability distribution when $n=35$ in Figure \ref{fig:nisthrityfiveSminusS}. We exhaustively listed every subset of $[n]$ and recorded the corresponding $|S-S|$. The probability distribution is exactly the frequencies divided by $2^{35}$.

\begin{figure}[h]
\center\includegraphics[width=0.8\textwidth]{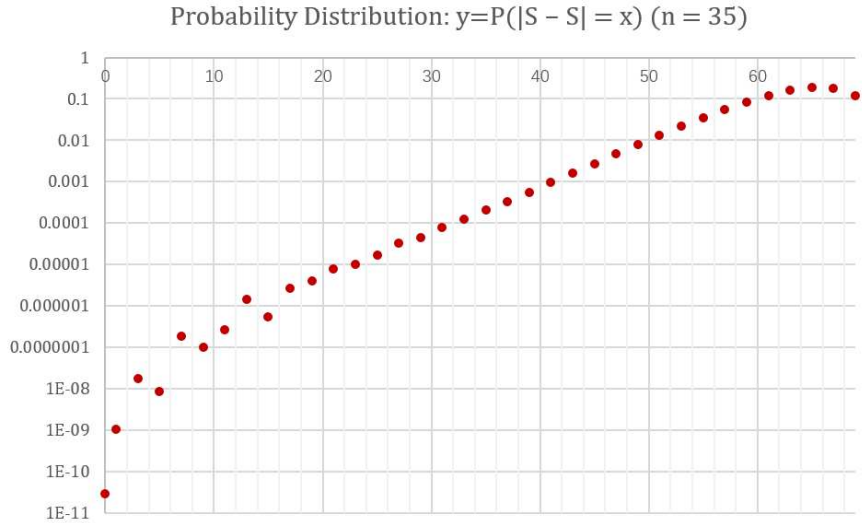}
\caption{\label{fig:nisthrityfiveSminusS} Probability distribution of $P(|S-S|=x)$ when $n=35$.}
\end{figure}

We make three observations from Figure \ref{fig:nisthrityfiveSminusS}.

\begin{itemize}
\item $|S-S|$ is either 0 or odd.
\item There are divots at having 5, 9 and 15 differences. That is,
\eq{
    P(|S-S|=3)&\ >\ P(|S-S|=5) \ <\ P(|S-S|=7),\nonumber\\
    P(|S-S|=7)& \ > \ P(|S-S|=9)\ < \ P(|S-S|=11),\tand\nonumber\\
    P(|S-S|=13)&\ > \ P(|S-S|=15)\ < \ P(|S-S|=17).
}

\item There is a peak at ``missing 4'' differences. That is,
\eq{\forall k\neq 4, P(2n-1-|S-S|=4) \ > \ P(2n-1-|S-S|=k)}
(thus when $n=35$, this is saying $|S-S|=65$ is the most likely cardinality of the diffset).
\end{itemize}

These observations seem to continue to hold for larger $n$, though our investigations are no longer exhaustive but instead are random samples from the space.

The first observation is trivial after realizing that if $m\in S-S$ then $-m\in S-S$.

For conciseness, let \eq{\pnh k \ := \ P(|S-S|=k), \ \ \ \pnm k \ :=\ P(2n-1-|S-S|=k).} Here, H means \textit{having} differences whereas M means \textit{missing}. They are two complementary perspectives.

\subsection{Main results}\

We prove that Observations 2 and 3 are true for sufficiently large $n$.

\begin{thm}\label{thm:propobs2}
Observation 2 is true for $n\ge 12$. That is, $\forall n\ge 12,$ \eq{&\pnh 3\ >\ \pnh 5\ < \ \pnh 7,\nonumber\\&\pnh 7 \ >\ \pnh 9\ <\ \pnh{11},\nonumber\\ \tand& \pnh{13} \ > \ \pnh{15}\ < \ \pnh{17}.} (Note that when $n=11$, Observation 2 fails because $\pnh{13}=269<275=\pnh{15}$.)
\end{thm}

\begin{thm}\label{thm:obs3}
Observation 3 is true for sufficiently large $n$. That is, \eq{\exists N:\forall n\ge N, \forall k\neq 4:\  \pnm 4\ >\ \pnm k.}
(Note that when $n=14$, Observation 3 fails because $\pnm 4=\pnm 2$. We don't know if this will ever happen again for larger $n$.)
\end{thm}

Similar to Theorem 1.9 in \cite{LMO}, we have the following result, which is used to prove Theorem \ref{thm:obs3}.

\begin{thm}\label{thm:l10}
The limiting probability distribution of missing differences,\newline $\ell(k):=\lin\pnm k$, is well-defined, positive on (and only on) even $k$'s, adds up to $1$, and satisfies \eq{\ell(10)\ < \ \ell(8)\ < \ \ell(0)\ < \ \ell(6)\ < \ \ell(2)\ < \ \ell(4).}
\end{thm}

Rigorous bounds for $\ell(k)$ are given in Theorem \ref{thm:lreal}. As a corollary, we provide an asymptotic bound for the OEIS sequence A103295, which counts the number of complete rulers\footnote{See Definition \ref{defn:compruler}.}.

\begin{thm}\label{thm:103295}
The OEIS sequence A103295 satisfies $a_n\sim c\cdot2^{n}$, where $0.2433<c<0.2451$.
\end{thm}

\section{Results about having (few) differences}

We give a few straightforward results on having few differences.

\begin{defi} A sequence $Q$ has a divot at $i$ if $Q_i$ is smaller than the nearest non-zero neighbor on each side of the sequence.\end{defi}

Note in the above definition we require the neighbors to be non-zero; this is important as the cardinalities of the number of missing differences is always even.

\begin{prop}\label{thm:divot5} For all $n\ge 4$, $\pnhz$ has a divot at 5: $\pnh 3 > \pnh 5 <\pnh 7$.\end{prop}

\begin{proof} We have the following characterizations, where we abbreviate a set $S$ is an arithmetic progression\footnote{This means there are integers $a, d$ and $m$ such that $S = \{a, a + d, a + 2d, \dots, a + md\}$.} by writing $S$ is an AP.
\begin{itemize}
    \item $|S-S|=3\Longleftrightarrow |S|=2$.
    \item $|S-S|=5\ifff |S|=3\tand S\text{ is an AP (e.g., }\{3,8,13\}\text{)}$.
    \item $|S-S|=7\ifff |S|=3\tand S\text{ is not an AP,} \tor |S|=4\tand S\text{ is an AP}$.
\end{itemize}

Thus, by counting arithmetic progressions, the following equations hold: \eq{2^n \pnh 3& \ = \ {n\choose 2}\nonumber\\2^n \pnh 5& \ = \ {\lfloor\novt\rfloor \choose 2}+{\lfloor\frac{n+1}{2}\rfloor \choose 2}\nonumber\\ 2^n \pnh 7& \ = \ {n\choose 3} - 2^n \pnh 5 + \su_{i=0}^2{\lfloor\frac{n+i}{3}\rfloor \choose 2}.} When $n\ge 4$, we have \eq{\pnh 3 \ > \ \pnh 5 \ \le\ \frac{{n\choose 3}}{2^n} - \pnh 5\ <\ \pnh 7.}
\end{proof}

In view of the proof, for any $k$ we see that $\pnh k$ can be written in a closed form in terms of $n$. Straightforward analysis shows the following.

\begin{prop} For all $n\ge 7$, $\pnhz$ has a divot at 9.\end{prop}

\begin{prop} For all $n\ge 12$, $\pnhz$ has a divot at 15.\end{prop}

The above allows us to conclude Theorem \ref{thm:propobs2}. \hfill $\Box$



\section{Results about missing (few) differences}

\subsection{Intuitively measuring the limiting probabilities}\


We show that the limiting probability of having $k$ differences, and that of missing $k$ differences, exist. The latter (Claim \ref{thm:pnmlim}) is a special case of Theorem 1.3 in \cite{Zh2}, but as some parts of this argument will be used later, we provide details.

\begin{cla}\label{thm:pnhlim} For all $k\ge 0$, $\lin\pnh k=0$.\end{cla}

\begin{proof} The claim follows immediately by noting $P(|S-S|=k)\le P(|S|\le k)\to 0$. \end{proof}


\begin{cla}\label{thm:pnmlim} For all $k\ge 0$, $\lin\pnm k$ exists and $\sui0\lin\pnm i = 1$.\end{cla}

\begin{proof} Recall Observation 1: when $k$ is odd, for all $n\neq\frac{k+1}{2}$ we have $\pnm k=0$. We are interested in evens.

$\forall k\ge 0,\forall m>k,\forall \epsilon>0, \forall n>2m,\forall S\subseteq [n],$ if $\{0,\dots,n-m-1\}\subseteq S-S$, then
\smeq{|(S-S)\cap \{n-m, \dots, n-1\}| \ = \ m-k&
\ \ifff\ |(S-S)\cap \{0, \dots, n-1\}|\  = \ n-k\nonumber\\
&\ \ifff\ |S-S|\ = \ 2n-1-2k.}
Thus \eq{{}&\left|\pnm{2k} - P\left(\left|\left(S-S\right)\cap \{n-m, \dots, n-1\}\right|=m-k\right)\right|\nonumber\\
\ \le\ &P(\{0,\dots,n-m-1\}\subsetneq S-S).}
The main term is constant with respect to $n$:
\smeq{
    {}&P\left(|(S-S)\cap \{n-m, \dots, n-1\}|=m-k\right)\nonumber\\
    =\ \ {}&P\left(\left|\left(\left(S\cap \{n-m,\dots,n-1\}\right)-\left(S\cap\{0,\dots,m-1\}\right)\right)\cap \{n-m, \dots, n-1\}\right|=m-k\right)\nonumber\\
    =\ \ {}&P_{S_1\subseteq [n]\setminus (n-m), S_2\subseteq [m]}\left(\left|\left(S_1-S_2\right)\cap \{n-m, \dots, n-1\}\right|=m-k\right)\nonumber\\
    =\ \ {}&P_{S\subseteq [2m]}\left(\left|\left(S-S\right)\cap \{m, \dots, 2m-1\}\right|=m-k\right)\nonumber\\
    =:\ {}&f_k(m).}

By Lemma 11 in \cite{MO}\footnote{It states that if $A$ is a uniformly randomly chosen subset of $[n]$, then \eqn{P(k\notin S-S)\begin{cases}\le \tvfp^{n/3} & 1\le k \le \novt\\ \le \tvfp^{n-k} & \novt \le k \le n-1.\end{cases}}}, \eq{
    P(\{0,\dots,n-m-1\}\subsetneq S-S)
    &\ \le\ \su_{i=0}^{n-m-1}P(i\notin S-S)\nonumber\\
    &\ \le\ \su_{i=0}^{\lfloor\novt\rfloor-1}\tvfp^\frac{n}{3} + \su_{i=\lfloor\novt\rfloor}^{n-m-1}\tvfp^{n-i}\nonumber\\
    &\ <\ \tvfp^\frac{n}{3}\cdot\novt +\tvfp^{m+1}\cdot 4\nonumber\\
    &\ <\ \epsilon+4\tvfp^{m+1}\text{ for sufficiently large }n.}

For sufficiently large $n$,
\be\label{eq:pnmbound}
    \left|\pnm{2k}-f_k(m)\right|\ < \ \epsilon+4\tvfp^{m+1}.
\ee

By the arbitrariness of $m\tand\epsilon$, $\{\pnm{2k}\}_n$ is Cauchy and so converges. The rest of the claim follows from non-negativity of the limits and the fact $\su_{i=0}^{2n-1}\pnm i=1$.
\end{proof}

\begin{rek}\label{thm:unicorn} Note $m>k$ is not needed, and since the bounded error, $\epsilon+4\tvfp^{m+1}$, is irrelevant to $k$, the convergence is uniform.
		
\end{rek}

\begin{defi} Let
$\ell(k):=\lin\pnm k$.
\end{defi}

\begin{lem}\label{thm:lovt}
For all $k\ge 0$, we have $\ell(2k+2)\ \ge\ \ell(2k)/2$.
\end{lem}

\begin{proof} We have
\eq{
{}&\pnm{2k+2}\nonumber\\
\ = \ &P(|S-S|=2n-1-2(k+1))\nonumber\\
\ \ge\ &P(n-1\notin S\ \land\  |(S-S)\cap\{-n+2,\dots,n-2\}|=2n-1-2(k+1))\nonumber\\
\ = \ &P(n-1\notin S)\cdot P_{S\subseteq [n-1]}(|S-S|=2(n-1)-1-2k)\nonumber\\
\ = \ &\oovt\pnmb{n-1}(2k).
}
Note the left and right hand sides converge to $\ell(2k+2)$ and $\ell(2k)/2$ respectively.
\end{proof}

\begin{cor}\label{thm:pnmevenpos}
For all $k\ge 0$, $\lin\pnm{2k}>0$.
\end{cor}

Compared with the distribution of having-differences (Claim \ref{thm:pnhlim}), this shows that the direction we view matters. We see non-zero limits at this end. 

\begin{rek}
By Remark \ref{thm:TT},\begin{alignat}{3}\pnmb{36}(0)\ = \ {}&& \frac{8342197304}{2^{36}}&\ \approx \  0.1214,\nonumber\\ \pnmb{36}(2)\ = \ {}&&\frac{12668987317
}{2^{36}}&\ \approx \  0.1843,\nonumber\\ \pnmb{36}(4)\ = \ {}&& \frac{12894355828}{2^{36}}&\ \approx \  0.1876, \nonumber\\ \pnmb{36}(6)\ = \ {}&& \frac{10879185718}{2^{36}}&\ \approx \  0.1583, \nonumber\\ \pnmb{36}(8)\ = \ {}&& \frac{8208838614}{2^{36}}&\ \approx \  0.1195.\end{alignat} This gives us a sensible (but not rigorous) estimate of $\ell(k)$.
\end{rek}

We do have a rigorous bound of $\ell(k)$, in view of the proof of Claim \ref{thm:pnmlim}.

\begin{prop}\label{thm:lbound}
For all $m>k$, $\left|\ell(2k)-f_k(m)\right|\le 4(\tovf)^{m+1}$.
\end{prop}

\begin{proof}
Replace $\pnm{2k}$ by $\ell(2k)$ in equation (\ref{eq:pnmbound}).
\end{proof}

One would like to use this fact to prove Theorem \ref{thm:l10}, since $f_k(m)$ is finitely computable. Unfortunately this quickly becomes unrealistic because it takes $4^{m}m^2$ computations to exhaustively determine $f_k(m)$, and to reduce the uncertainty to $(0.1876-0.1843)/2$ we should have $m\ge 27$. In 2019, it took our laptop\footnote{CPU: i7-6500U @ 2.5GHz, RAM: 8GB} around 5 minutes to run $m=17$ with this method, and thus it would need around 25.2 years to computationally verify the theorem. We thus need a better approach, which we describe below.

\subsection{Using Conditional Probabilities}   

\begin{lem}\label{thm:localis} The conditional probability of $k\not\in S-S$, given that $0, n-1\in S$, is bounded by the following:
\be\threecase
{P\left(k\notin S-S\cht\right)\ }
{\ = \ 0}{k=n-1}
{\ = \ \fovn\cdot\tvfp^{n-k}}{\novt\le k<n-1}
{\ \le\ \fovn\cdot\tvfp^{\frac{n}{3}}}{0\le k<\novt.}\ee
\end{lem}

\begin{proof}
For all $k < n$ let $D:=\left\{\{a,b\}: a,b\in [n], |a-b|=k\right\}$. We say $D'\subseteq D$ is \textit{mutually disjoint} if $\forall p_1,p_2\in D'$, $p_1\cap p_2=\emptyset$. If $D'\subseteq D$ is mutually disjoint and $0,n-1\in\bigcup D'$ (the union is over all the pairs in $D'$), then
\eq{
    P\left(k\notin S-S\cht\right)
    &\ = \ P\left(D\cap\mathcal{P}(S) = \emptyset\cht\right)\nonumber\\
    &\ \le\ P\left(D'\cap\mathcal{P}(S) = \emptyset\cht\right)\nonumber\\
    &\ = \ \prod_{p\in D'}{\left(1-2^{-|p\setminus\{0,n-1\}|}\right)}\nonumber\\
    &\ = \ \twocase{}{0}{k=n-1}{\fovn\cdot\tvfp^{|D'|}}{0\le k<n-1.}
}

When $2k>n-1$, $D$ is already mutually disjoint and has size $n-k$; otherwise, we can find a mutually disjoint $D'$ with $|D'|\ge n/3$, and let $0,n-1\in\bigcup D'$ without loss of generality. We hence conclude the lemma.
\end{proof}

The conditional probability distribution requiring $0,n-1\in S$ is compared with the usual probability distribution without such restriction. We define similar notions to $\pnmz, f_k$.

\begin{defi} Let
\eqn{\qnmb{n}(k)& \ := \ P\left(\misk\cht\right);\nonumber\\  g_k(m)& \ :=\ P_{S\subseteq [2m]}\left(|(S-S)\cap \{m, \dots, 2m-1\}|\ = \ m-k\cht[0,2m-1\in]\right).}
\end{defi}

\begin{prop}
$\forall k\ge 0, \forall m>k,\forall \epsilon>0$ and for sufficiently large $n$, \[\left|\qnmb{n}(2k)-g_k(m)\right| \ < \ \epsilon+\frac{16}{9}\cdot\tvfp^{m+1}.\]
\end{prop}
\begin{proof}
This follows from an analagous argument as in Claim \ref{thm:pnmlim}. By Lemma \ref{thm:localis}, the uncertainty is 4/9 the original one.
\end{proof}

\begin{defi} We have
$j(k):=\lin \qnmb{n}(k)$.
\end{defi}

\begin{prop}\label{thm:jbound} Note
$j(k)$ is well-defined; in addition, for all $m>k$ we have $$ |j(2k) - g_k(m)| \ < \ \frac{16}{9} \tvfp^{m+1}. $$
\end{prop}

\begin{proof}
The proof is similar to that of Proposition \ref{thm:lbound}.
\end{proof}


\begin{lem}
For \kitn, \[\ell(k)\ = \ \frac{j(k)}{4}+\ell(k-2)-\frac{\ell(k-4)}{4}.\]
\end{lem}

\begin{proof}
\eq{
    \pnm k
    &\ = \ P(\misk)\nonumber\\
        &\ = \ \oovf P\left(\misk\cht\right)\nonumber\\
        &\ \ \ \ \ \ + \oovt P\left(\misk\cht[0\notin]\right)\nonumber\\
        &\ \ \ \ \ \ + \oovt P\left(\misk\cht[n-1\notin]\right) \nonumber\\
        &\ \ \ \ \ \  - \oovf P\left(\misk\cht[0,n-1\notin]\right)\nonumber\\
    &\ = \ \oovf \qnmb{n}(k) + \oovt \pnmb{n-1}(k-2) + \oovt \pnmb{n-1}(k-2) - \oovf\pnmb{n-2}(k-4).
}
The left and right hand sides converge to $\ell(k)$ and $\frac{j(k)}{4}+\ell(k-2)-\frac{\ell(k-4)}{4}$ respectively.
\end{proof}

\begin{cor}\label{thm:jform}
For \kitn, \[j(k)\ = \ 4\ell(k)-4\ell(k-2)+\ell(k-4),\tand \ell(k)\ = \ \sui{0}\frac{i+1}{2^{i+2}}j(k-2i).\]
\end{cor}

\begin{cor}\label{thm:ldiff}
For \kitn, \[\ell(k)-\ell(k+2)\ = \ -\oovf j(k+2)+\sui{1}\frac{i}{2^{i+3}}j(k-2i).\]
\end{cor}

\begin{rek}
It's better to focus on and compute the $j$ sequence than the $\ell$ sequence, for the following reasons.
\begin{itemize}
    \item Using the same value of $m$, estimating the $j$ sequence will produce less uncertainty than estimating the $\ell$ sequence. In view of Proposition \ref{thm:lbound} and Proposition \ref{thm:jbound}, given $f_k(m)$ and $g_k(m)$, which are finitely computable, $\ell(2k)$ is within $4\tvfp^{m+1}$ from $f_k(m)$, while $j(2k)$ is within only $\frac{16}{9}\tvfp^{m+1}$ from $g_k(m)$, reducing to a factor of $4/9$.

    \item When estimating $\ell(2)-\ell(4)$, which is the bottleneck difference regarding Theorem \ref{thm:obs3}, the uncertainty coming from the $j$ sequence would be further compressed while that from $\ell$ would be amplified. Say each term in the $j$ sequence has an uncertainty of $e$, then by Corollary \ref{thm:ldiff}, the uncertainty of $\ell(2)-\ell(4)$ is only $(\oovf+\frac{1}{16})e=5e/16$, whereas if we estimated the $\ell$ sequence honestly the uncertainty would be $2e$.\footnote{The bottleneck difference for Theorem \ref{thm:l10} is $\ell(0)-\ell(8)$, which would have uncertainty $73e/64$ under the $j$ method by Corollary \ref{thm:jform}, but $2e$ under the $\ell$ method.}

    \item What's more, it is 4x faster to compute $g_k(m)$ than $f_k(m)$ because the conditional probability reduces two degrees of freedom.
\end{itemize}
Approximately\footnote{This is a rough estimate: the computational complexities of $f_k(m)$ and $g_k(m)$ are both asymptotically $4^m\cdot m^2$, but when $m$ is decreased we only counted the boost coming from the $4^m$ factor, neglecting that from the quadratic term; also, $m$ is always an integer, so there are floor-and-ceiling errors.}, the $j$ method is $4^{\log_{3/4}(\fovn\cdot\frac{5}{32})}\times 4\approx 1527656$ times faster than the $\ell$ method to verify Theorem \ref{thm:obs3}, and $\approx 2981$ times faster to verify Theorem \ref{thm:l10}. One can divide the 25.2 years (mentioned earlier) by these numbers to see how everything is going to become feasible.
\end{rek}



Armed with these results, we are ready now to prove Theorem \ref{thm:l10}.

\subsection{Calculations and results}

\begin{calc}\label{thm:gkvalscalc}
The code in Appendix \ref{app:code} calculates the data in Table \ref{tab:gkvals}.
\end{calc}

\begin{table}[!ht]
    \centering\small
    \begin{tabular}{c r}
        $k$ & $g_k(23)$ \\ \hline
        0 & $8592305829704/2^{44}$ \\
        1 & $4442759682300/2^{44}$ \\
        2 & $2367846591103/2^{44}$ \\
        3 & $1174068145740/2^{44}$ \\
        4 & $559669653171/2^{44}$ \\
        5 & $256031157923/2^{44}$ \\
        6 & $114186380080/2^{44}$ \\
        7 & $49736070308/2^{44}$ \\
        8 & $21123843993/2^{44}$ \\
        9 & $8778930083/2^{44}$ \\
        10 & $3543398884/2^{44}$ \\
        11 & $1378772067/2^{44}$ \\
        12 & $508048560/2^{44}$ \\
        13 & $174732658/2^{44}$ \\
        14 & $54900922/2^{44}$ \\
        15 & $15344643/2^{44}$ \\
        16 & $3692910/2^{44}$ \\
        17 & $737437/2^{44}$ \\
        18 & $116855/2^{44}$ \\
        19 & $13885/2^{44}$ \\
        20 & $1134/2^{44}$ \\
        21 & $55/2^{44}$ \\
        22 & $1/2^{44}$ \\
    \end{tabular}
    \caption{Values of $g_k(m)$ when $m=23$.}
    \label{tab:gkvals}
\end{table}

\begin{lem}\label{thm:rigboundl} The following inequalities hold:
\eq{
    \ell(0)-\ell(2)&\in (-0.06359,-0.06268)\nonumber\\
    \ell(2)-\ell(4)&\in (-0.00369,-0.00256) \nonumber\\
    \ell(4)-\ell(6)&\in (0.02895, 0.03030) \nonumber\\
    \ell(6)-\ell(8)&\in (0.03838,0.03989) \nonumber\\
    \ell(8)-\ell(10)&\in (0.03523, 0.03686).
}
In particular, \eq{\ell(10) \ < \ \ell(8)\ < \ \ell(0)\ < \ \ell(6)\ <\ \ell(2)\ < \ \ell(4).}\end{lem}

\begin{proof}
This follows from Proposition \ref{thm:jbound}, Corollary \ref{thm:ldiff} and Calculation \ref{thm:gkvalscalc}.
\end{proof}

\begin{proof}[Proof of Theorem \ref{thm:l10}]
Follows from Claim \ref{thm:pnmlim}, Corollary \ref{thm:pnmevenpos} and Lemma \ref{thm:rigboundl}.
\end{proof}

\ \\

We report on some numerical bounds.

\begin{thm}\label{thm:lreal} The following inequalities hold:
\eq{0.12165&\ < \ \ell(0)\ < \ 0.12255 \nonumber\\
0.18434&\ < \ \ell(2)\ < \ 0.18614  \nonumber\\
0.18713&\ < \ \ell(4)\ < \ 0.18959 \nonumber\\
0.15728&\ < \ \ell(6)\ < \ 0.16019  \nonumber\\
0.11801&\ < \ \ell(8)\ < \ 0.12119 \nonumber\\
0.08188&\ < \ \ell(10)\ < \ 0.08523 \nonumber\\
0.05355&\ < \ \ell(12)\ < \ 0.05700 \nonumber\\
0.03334&\ < \ \ell(14)\ < \ 0.03685 \nonumber\\
0.01981&\ < \ \ell(16)\ < \ 0.02335 \nonumber\\
0.01115&\ < \ \ell(18)\ < \ 0.01471 \nonumber\\
0.00580&\ < \ \ell(20)\ < \ 0.00937.}
\end{thm}

\begin{proof}
The claims follow from Proposition \ref{thm:jbound}, Corollary \ref{thm:jform} and Calculation \ref{thm:gkvalscalc}.
\end{proof}


The rigorous bounds are illustrated in Figure \ref{fig:lkbound}.

\begin{figure}[h]
\center\includegraphics[width=0.8\textwidth]{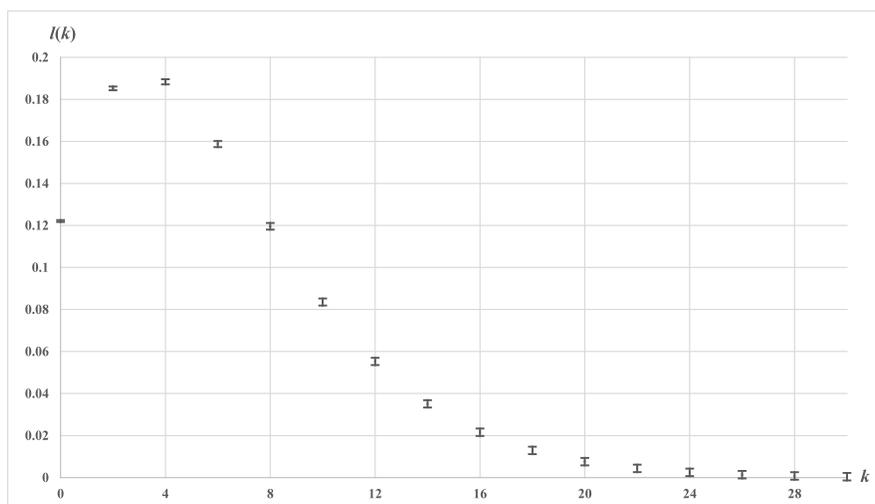}
\caption{\label{fig:lkbound} Bounds of $\ell(k)$ for $0\le k\le 30$. (Odd $k$'s are omitted.)}
\end{figure}

After proving an auxiliary result we will prove Theorem \ref{thm:obs3}.

\begin{lem}\label{thm:el}
$\sui{0} i \cdot \ell(i)=6$.
\end{lem}

\begin{proof}
\eq{
    \sui{0}i \cdot \ell(i)&\ = \ \lin\sui{0}i \cdot\pnm i\nonumber\\
    &\ = \ \lin\sui{0}\frac{1}{2^n}(2n-1-(2n-1-i))\cdot\#(S\subseteq [n]: |S-S|=  2n-1-i)\nonumber\\
    &\ = \ \lin(\frac{1}{2^n}(2n-1)2^n-\frac{1}{2^n}\sui{0}i\#(S\subseteq [n]: |S-S|=i))\nonumber\\
    &\ = \ \lin(2n-1-\frac{1}{2^n}\su_{S\subseteq [n]}|S-S|)\nonumber\\
    &\ = \ 6\text{ (by Theorem 3 of \cite{MO}\footnotemark{}).}
}
\footnotetext{It states that for any AP $A$ of size $n$, $\frac{1}{2^n}\sum_{S\subseteq A}|S-S|$ converges to $2n - 7$ when $n\to \infty$.}
\end{proof}



\begin{thm}\label{thm:lpeak}
For all $k\neq 4$, $\ell(k)<\ell(4)$.
\end{thm}

\begin{proof}
Theorem \ref{thm:l10} proves the case for $k<12$. When $k\ge12$, by Lemma \ref{thm:lovt} and \ref{thm:el},
\eq{
    2(k-6)  \cdot  \ell(k)&\ <\ \sui{k}(k-6) \cdot \ell(i)\nonumber\\
    &\ <\ \sui{6}(i-6)  \cdot \ell(i)\nonumber\\
    &\ = \ \sui{0}(i-6) \cdot \ell(i)+6\ell(0)+4\ell(2)+2\ell(4)\nonumber\\
    &\ <\ \sui{0}i \cdot \ell(i)-6\sui{0}\ell(i)+(6+4+2)\ell(4)\nonumber\\
    &\ = \ 12 \ell(4).
}
						
Thus $\Rightarrow \ell(k)<\ell(4)$.
\end{proof}

\begin{proof}[Proof of Theorem \ref{thm:obs3}]
The theorem follows from Theorem \ref{thm:lpeak} and Remark \ref{thm:unicorn}.
\end{proof}

\begin{rek}
Theorem \ref{thm:obs3} gives a partial answer to Question \ref{que:311}; the rather strange occurrence of $\pnm 2=\pnm 4$ happens only finitely many times.
\end{rek}

\subsection{About rulers}

\begin{defi}\label{defn:compruler}
A \textit{ruler} of length $L$ is any subset $R\subseteq \{0, \dots, L\}$. It is \textit{complete} if it can measure every distance shorter or equal to its length; that is, $\{0,...,L\}\subseteq R-R$.
\end{defi}

\begin{lem}\label{thm:rulers}
Let $a_n$ be the number of complete rulers of length $n$; then $a_{n-1}\sim \ell(0)\cdot 2^n$.
\end{lem}
\begin{proof}
$S\subseteq[n]$ is a complete ruler of length $n-1$ iff $|S-S|=2n-1$, so the number of complete rulers of length $n-1$ is equal to $\pnm 0\cdot 2^n$, which goes to $\ell(0)\cdot 2^n$.
\end{proof}

\begin{proof}[Proof of Theorem \ref{thm:103295}]
The claim follows from Lemma \ref{thm:rulers} and Theorem \ref{thm:lreal}. Here~$c=2\ell(0)$.
\end{proof}

\section{Conjectures}

Intuitively, when $k\lll n$, randomly choosing $k$ elements from $[n]$ usually gives $|S-S|=k(k-1)+1$. On the other hand, to have $|S-S|=k(k-1)+3$ requires a maximal appearance of coincidences (repeated differences). Hence we have the following conjecture about the divots in $\pnhz$.

\begin{conj} For every $k>1$, $k(k-1)+3$ is a divot of $\pnhz$ for sufficiently large $n$. Furthermore, they are the only divots.\end{conj}

We also noticed that once a divot appears in $\pnhz$, it seems to never move again:

\begin{conj} If $k$ is a divot of $\pnhz$ for $n=n_1$, then it is also a divot for any $n> n_1$.\end{conj}

About missing differences, we proved Theorem \ref{thm:obs3} by limits, hence not giving an explicit threshold $N$ such that every $n\ge N$ satisfies Observation 3. Experimental data suggest that 15 might be enough already, so we guess:

\begin{conj}
For all $n\ge 15$, $\forall k \neq 4$, $\pnm 4 > \pnm k$.
\end{conj}

Recall that in Theorem \ref{thm:l10}, we compared the limiting probabilities of missing 0, 2, 4, 6, 8 and 10 differences, and found no divot. What about missing 12, or more? In fact, any two limiting probabilities can be approximated to be arbitrarily precise using our method, but we couldn't bound infinite many of them at the same time. Both intuition and experimental data seem to suggest that the decay after $\ell(4)$ should go on forever. Thus, we leave the following conjecture.

\begin{conj}\label{conj:idealconj}
In fact, $\ell(4)  > \ell(2)  > \ell(6) > \ell(0)  > \ell(8)  > \ell(10)  > \ell(12)  > \cdots$. In other words, the sequence $\ell$ has no divots.
\end{conj}

\FloatBarrier
\newpage
		
\appendix
\section{Distribution of $|S-S|$ when $n\le 36$}\label{appendix:distrSminusSthirtysix}

\begin{table}[!h]\tiny
  \caption{Number of $S\subseteq [n]$ with $|S-S|=k$. ($n\le 24$)}
  \resizebox{\textwidth}{!} {
  \begin{tabular}{c | *{25}{r}}
      \diagbox{k}{n} & 0 & 1 & 2 & 3 & 4 & 5 & 6 & 7 & 8 & 9 & 10 & 11 & 12 & 13 & 14 & 15 & 16 & 17 & 18 & 19 & 20 & 21 & 22 & 23 & 24 \\ \hline 0 & 1 & 1 & 1 & 1 & 1 & 1 & 1 & 1 & 1 & 1 & 1 & 1 & 1 & 1 & 1 & 1 & 1 & 1 & 1 & 1 & 1 & 1 & 1 & 1 & 1 \\ 1 & 0 & 1 & 2 & 3 & 4 & 5 & 6 & 7 & 8 & 9 & 10 & 11 & 12 & 13 & 14 & 15 & 16 & 17 & 18 & 19 & 20 & 21 & 22 & 23 & 24 \\ 3 & 0 & 0 & 1 & 3 & 6 & 10 & 15 & 21 & 28 & 36 & 45 & 55 & 66 & 78 & 91 & 105 & 120 & 136 & 153 & 171 & 190 & 210 & 231 & 253 & 276 \\ 5 & 0 & 0 & 0 & 1 & 2 & 4 & 6 & 9 & 12 & 16 & 20 & 25 & 30 & 36 & 42 & 49 & 56 & 64 & 72 & 81 & 90 & 100 & 110 & 121 & 132 \\ 7 & 0 & 0 & 0 & 0 & 3 & 8 & 17 & 31 & 51 & 77 & 112 & 155 & 208 & 272 & 348 & 436 & 539 & 656 & 789 & 939 & 1107 & 1293 & 1500 & 1727 & 1976 \\ 9 & 0 & 0 & 0 & 0 & 0 & 4 & 10 & 17 & 27 & 43 & 62 & 85 & 113 & 148 & 189 & 236 & 289 & 352 & 423 & 501 & 588 & 687 & 795 & 913 & 1042 \\ 11 & 0 & 0 & 0 & 0 & 0 & 0 & 9 & 25 & 47 & 77 & 113 & 170 & 237 & 319 & 413 & 531 & 666 & 825 & 1000 & 1206 & 1430 & 1691 & 1970 & 2289 & 2630 \\ 13 & 0 & 0 & 0 & 0 & 0 & 0 & 0 & 17 & 49 & 97 & 169 & 269 & 409 & 606 & 863 & 1195 & 1607 & 2115 & 2735 & 3492 & 4393 & 5450 & 6690 & 8130 & 9790 \\ 15 & 0 & 0 & 0 & 0 & 0 & 0 & 0 & 0 & 33 & 93 & 177 & 275 & 402 & 549 & 730 & 967 & 1238 & 1562 & 1932 & 2355 & 2829 & 3345 & 3946 & 4613 & 5343 \\ 17 & 0 & 0 & 0 & 0 & 0 & 0 & 0 & 0 & 0 & 63 & 187 & 377 & 629 & 973 & 1417 & 1978 & 2688 & 3628 & 4765 & 6151 & 7794 & 9781 & 12089 & 14774 & 17861 \\ 19 & 0 & 0 & 0 & 0 & 0 & 0 & 0 & 0 & 0 & 0 & 128 & 377 & 747 & 1228 & 1850 & 2642 & 3633 & 4849 & 6340 & 8278 & 10580 & 13381 & 16603 & 20474 & 24909 \\ 21 & 0 & 0 & 0 & 0 & 0 & 0 & 0 & 0 & 0 & 0 & 0 & 248 & 747 & 1509 & 2507 & 3770 & 5338 & 7271 & 9641 & 12469 & 15909 & 20315 & 25533 & 31893 & 39392 \\ 23 & 0 & 0 & 0 & 0 & 0 & 0 & 0 & 0 & 0 & 0 & 0 & 0 & 495 & 1472 & 2975 & 4999 & 7519 & 10654 & 14499 & 19129 & 24681 & 31221 & 38903 & 48354 & 59263 \\ 25 & 0 & 0 & 0 & 0 & 0 & 0 & 0 & 0 & 0 & 0 & 0 & 0 & 0 & 988 & 2975 & 6022 & 10104 & 15278 & 21596 & 29249 & 38430 & 49408 & 62377 & 77572 & 95318 \\ 27 & 0 & 0 & 0 & 0 & 0 & 0 & 0 & 0 & 0 & 0 & 0 & 0 & 0 & 0 & 1969 & 5911 & 11985 & 20192 & 30501 & 43062 & 58148 & 76121 & 97667 & 123155 & 153424 \\ 29 & 0 & 0 & 0 & 0 & 0 & 0 & 0 & 0 & 0 & 0 & 0 & 0 & 0 & 0 & 0 & 3911 & 11880 & 24103 & 40524 & 61350 & 86236 & 115893 & 150319 & 190510 & 236824 \\ 31 & 0 & 0 & 0 & 0 & 0 & 0 & 0 & 0 & 0 & 0 & 0 & 0 & 0 & 0 & 0 & 0 & 7857 & 23734 & 48377 & 81542 & 123470 & 174352 & 234160 & 304245 & 385858 \\ 33 & 0 & 0 & 0 & 0 & 0 & 0 & 0 & 0 & 0 & 0 & 0 & 0 & 0 & 0 & 0 & 0 & 0 & 15635 & 47474 & 96676 & 162994 & 246765 & 347050 & 465537 & 602109 \\ 35 & 0 & 0 & 0 & 0 & 0 & 0 & 0 & 0 & 0 & 0 & 0 & 0 & 0 & 0 & 0 & 0 & 0 & 0 & 31304 & 94885 & 193562 & 326913 & 494449 & 696108 & 931109 \\ 37 & 0 & 0 & 0 & 0 & 0 & 0 & 0 & 0 & 0 & 0 & 0 & 0 & 0 & 0 & 0 & 0 & 0 & 0 & 0 & 62732 & 190623 & 388606 & 656644 & 993569 & 1396647 \\ 39 & 0 & 0 & 0 & 0 & 0 & 0 & 0 & 0 & 0 & 0 & 0 & 0 & 0 & 0 & 0 & 0 & 0 & 0 & 0 & 0 & 125501 & 380805 & 776640 & 1312446 & 1985532 \\ 41 & 0 & 0 & 0 & 0 & 0 & 0 & 0 & 0 & 0 & 0 & 0 & 0 & 0 & 0 & 0 & 0 & 0 & 0 & 0 & 0 & 0 & 250793 & 763402 & 1557467 & 2633237 \\ 43 & 0 & 0 & 0 & 0 & 0 & 0 & 0 & 0 & 0 & 0 & 0 & 0 & 0 & 0 & 0 & 0 & 0 & 0 & 0 & 0 & 0 & 0 & 503203 & 1528095 & 3117611 \\ 45 & 0 & 0 & 0 & 0 & 0 & 0 & 0 & 0 & 0 & 0 & 0 & 0 & 0 & 0 & 0 & 0 & 0 & 0 & 0 & 0 & 0 & 0 & 0 & 1006339 & 3061916 \\ 47 & 0 & 0 & 0 & 0 & 0 & 0 & 0 & 0 & 0 & 0 & 0 & 0 & 0 & 0 & 0 & 0 & 0 & 0 & 0 & 0 & 0 & 0 & 0 & 0 & 2014992 \\ 49 & 0 & 0 & 0 & 0 & 0 & 0 & 0 & 0 & 0 & 0 & 0 & 0 & 0 & 0 & 0 & 0 & 0 & 0 & 0 & 0 & 0 & 0 & 0 & 0 & 0 \\
  \end{tabular}
  }
\end{table}
\begin{table}[!h]\tiny
  \caption{Number of $S\subseteq [n]$ with $|S-S|=k$. ($24\le n\le 36$)}
  \resizebox{\textwidth}{!} {
  \begin{tabular}{c | *{13}{r}}
      \diagbox{k}{n} & 24 & 25 & 26 & 27 & 28 & 29 & 30 & 31 & 32 & 33 & 34 & 35 & 36 \\ \hline 0 & 1 & 1 & 1 & 1 & 1 & 1 & 1 & 1 & 1 & 1 & 1 & 1 & 1 \\ 1 & 24 & 25 & 26 & 27 & 28 & 29 & 30 & 31 & 32 & 33 & 34 & 35 & 36 \\ 3 & 276 & 300 & 325 & 351 & 378 & 406 & 435 & 465 & 496 & 528 & 561 & 595 & 630 \\ 5 & 132 & 144 & 156 & 169 & 182 & 196 & 210 & 225 & 240 & 256 & 272 & 289 & 306 \\ 7 & 1976 & 2248 & 2544 & 2864 & 3211 & 3584 & 3985 & 4415 & 4875 & 5365 & 5888 & 6443 & 7032 \\ 9 & 1042 & 1184 & 1338 & 1504 & 1682 & 1876 & 2084 & 2305 & 2541 & 2795 & 3064 & 3349 & 3651 \\ 11 & 2630 & 3010 & 3419 & 3876 & 4357 & 4886 & 5443 & 6060 & 6707 & 7410 & 8143 & 8940 & 9776 \\ 13 & 9790 & 11699 & 13868 & 16325 & 19094 & 22202 & 25674 & 29543 & 33832 & 38569 & 43786 & 49515 & 55787 \\ 15 & 5343 & 6158 & 7029 & 7980 & 9024 & 10164 & 11384 & 12696 & 14093 & 15597 & 17216 & 18941 & 20767 \\ 17 & 17861 & 21464 & 25554 & 30192 & 35439 & 41365 & 47972 & 55334 & 63485 & 72583 & 82597 & 93598 & 105615 \\ 19 & 24909 & 30034 & 35835 & 42560 & 50164 & 58778 & 68336 & 79218 & 91199 & 104572 & 119214 & 135569 & 153328 \\ 21 & 39392 & 48297 & 58729 & 70921 & 85023 & 101393 & 120236 & 141992 & 166842 & 195124 & 227418 & 263837 & 304894 \\ 23 & 59263 & 72166 & 86779 & 103803 & 122773 & 144495 & 168711 & 195948 & 226062 & 259777 & 297046 & 338522 & 383708 \\ 25 & 95318 & 116803 & 141545 & 170669 & 203518 & 241453 & 283954 & 332047 & 385486 & 445578 & 511668 & 585268 & 666132 \\ 27 & 153424 & 188936 & 230785 & 281634 & 340918 & 411385 & 492735 & 587687 & 696368 & 821738 & 964188 & 1126614 & 1309990 \\ 29 & 236824 & 290286 & 351743 & 422400 & 502848 & 598252 & 705828 & 831558 & 972438 & 1134483 & 1314383 & 1519559 & 1747229 \\ 31 & 385858 & 480260 & 589088 & 713474 & 855957 & 1018020 & 1202962 & 1419676 & 1664732 & 1947773 & 2265195 & 2627654 & 3032028 \\ 33 & 602109 & 759570 & 939048 & 1145157 & 1379205 & 1646202 & 1948206 & 2289594 & 2673659 & 3121284 & 3619723 & 4191609 & 4824889 \\ 35 & 931109 & 1202343 & 1512270 & 1865592 & 2266137 & 2720935 & 3236533 & 3821295 & 4483176 & 5231412 & 6075752 & 7058965 & 8161491 \\ 37 & 1396647 & 1867806 & 2404100 & 3013664 & 3697776 & 4468556 & 5330593 & 6293553 & 7368022 & 8567388 & 9903780 & 11391366 & 13047575 \\ 39 & 1985532 & 2792117 & 3726584 & 4795360 & 5994044 & 7342144 & 8845276 & 10520512 & 12382684 & 14456863 & 16757210 & 19313503 & 22151419 \\ 41 & 2633237 & 3984017 & 5596451 & 7469425 & 9586795 & 11966365 & 14608625 & 17543417 & 20782662 & 24369445 & 28318130 & 32680465 & 37482058 \\ 43 & 3117611 & 5270104 & 7970998 & 11195574 & 14913983 & 19131301 & 23822819 & 29022146 & 34739876 & 41039669 & 47936336 & 55509344 & 63800433 \\ 45 & 3061916 & 6244117 & 10557091 & 15968677 & 22417023 & 29862931 & 38239392 & 47566626 & 57804101 & 69047026 & 81288502 & 94666428 & 109216351 \\ 47 & 2014992 & 6125358 & 12494664 & 21122722 & 31935586 & 44822674 & 59651353 & 76346946 & 94783970 & 115036473 & 137031262 & 160950680 & 186816887 \\ 49 & 0 & 4035985 & 12278446 & 25038586 & 42321005 & 63983506 & 89749444 & 119386846 & 152607226 & 189351319 & 229343035 & 272803379 & 319629353 \\ 51 & 0 & 0 & 8080448 & 24564954 & 50090752 & 84658919 & 127967673 & 179465499 & 238552257 & 304816636 & 377630128 & 456991110 & 542473471 \\ 53 & 0 & 0 & 0 & 16169267 & 49200792 & 100303312 & 169496641 & 256144840 & 359073831 & 477185749 & 609113912 & 754212597 & 911317415 \\ 55 & 0 & 0 & 0 & 0 & 32397761 & 98478615 & 200765677 & 339187677 & 512453496 & 718291220 & 953949620 & 1217261287 & 1505590283 \\ 57 & 0 & 0 & 0 & 0 & 0 & 64826967 & 197164774 & 401837351 & 678805584 & 1025433250 & 1436715877 & 1907636501 & 2432498687 \\ 59 & 0 & 0 & 0 & 0 & 0 & 0 & 129774838 & 394536002 & 804070333 & 1358091161 & 2051059855 & 2873264810 & 3813305230 \\ 61 & 0 & 0 & 0 & 0 & 0 & 0 & 0 & 259822143 & 789993459 & 1609586119 & 2717986051 & 4104228068 & 5747795503 \\ 63 & 0 & 0 & 0 & 0 & 0 & 0 & 0 & 0 & 520063531 & 1580640910 & 3220331421 & 5437313809 & 8208838614 \\ 65 & 0 & 0 & 0 & 0 & 0 & 0 & 0 & 0 & 0 & 1040616486 & 3163602123 & 6444236200 & 10879185718 \\ 67 & 0 & 0 & 0 & 0 & 0 & 0 & 0 & 0 & 0 & 0 & 2083345793 & 6330608624 & 12894355828 \\ 69 & 0 & 0 & 0 & 0 & 0 & 0 & 0 & 0 & 0 & 0 & 0 & 4168640894 & 12668987317 \\ 71 & 0 & 0 & 0 & 0 & 0 & 0 & 0 & 0 & 0 & 0 & 0 & 0 & 8342197304 \\ 73 & 0 & 0 & 0 & 0 & 0 & 0 & 0 & 0 & 0 & 0 & 0 & 0 & 0 \\
  \end{tabular}
  }
\end{table}

\begin{rek}\label{thm:TT}
Denoting the table by $T$, $T_{n,k}/2^n = \pnh k = \pnm{2n-1-k}$.
\end{rek}


\begin{que}\label{que:311}
Observe that when $n=3,11,12,14$, $\pnm2=\pnm4$. Such frequent repetition of large numbers doesn't look so random. Is there any reason behind it? Will it happen again?
\end{que}

\FloatBarrier
\newpage
\section{Code for Estimating $j(2k)$}\label{app:code}

\begin{lstlisting}[language=C,basicstyle=\ttfamily, breaklines, prebreak={\textbackslash}, columns=fullflexible,showstringspaces=false,keepspaces=true,frame=single]
#include <stdio.h>
#include <time.h>
#include <math.h>
long long cnts[100];
int main() {
    int m = 23, d; // Measure prob of missing n-1, ... n-m diffs
    clock_t begin = clock();
    double j[100], error = pow(0.75, m+1) * 16 / 9;
    long long cnt, r1 = (1LL << 2*m), r2 = 2*m-1;
    for(long long S = (1LL << (2*m-1)) + 1; S < r1; S+=2) {
        cnt = 0;
        for(d = m; d < r2; d++) if(!(S & (S >> d))) cnt++;
        cnts[cnt]++;
    }
    for(int i=0; i<m; i++) {
        j[2*i] = 1.0 * cnts[i] / pow(2, 2*m-2);
        printf("j(%d) = %f+-%e\t(G%d = %lld/2^%d)\n", 2*i,
        j[2*i], error, i, cnts[i], 2*m-2);
    }
    clock_t end = clock();
    double time_spent = (double)(end - begin) / CLOCKS_PER_SEC;
    printf("\nj(0)/4 < (%f+%f)/4 = %f <? %f = %f-%f < j(4).\nj(0)+ j(2) > %f+%f-2*%f = %f >? %f = 4(%f+%f) = 4j(6).\n",
    j[0], error, j[0]/4+error/4, j[4]-error, j[4], error,
    j[0], j[2], error, j[0]+j[2]-2*error, j[6]*4+error*4,
    j[6], error);
    printf("In %f sec.\n", time_spent);
    return 0;
}

\end{lstlisting}

\begin{rek}
The algorithm is $\Omega(4^m m^2)$. When $m=23$, it runs for 92.73 hours on our laptop. In fact, even when $m=18$, which takes only 3 minutes to run, the results could already establish $\ell(2) - \ell(4)<0$, and hence Theorem \ref{thm:obs3}, although it's not strong enough to show that $\ell(0)>\ell(8)$. The reader is welcome to confirm our calculations or achieve better bounds.
\end{rek}


\bigskip

\begin{thebibliography}{{Zha}DKMMWW9}

\bibitem[AMMS]{AMMS}
M. Asada, S. Manski, S. J. Miller and H. Suh, \emph{Fringe pairs in generalized MSTD sets}, International Journal of Number Theory \textbf{13} (2017), no. 10, 2653--2675.

\bibitem[BELM]{BELM}
A. Bower, R. Evans, V. Luo and S. J. Miller, \emph{Coordinate sum and difference sets of $d$-dimensional modular hyperbolas}, INTEGERS \#A31, 2013, 16 pages.

\bibitem[CLMS]{CLMS}
H. Chu, N. Luntzlara, S. J. Miller and L. Shao, \emph{Generalizations of a Curious Family of MSTD Sets Hidden By Interior Blocks}, to appear in Integers.

\bibitem[CMMXZ]{CMMXZ}
H. Chu, N. McNew, S. J. Miller, V. Xu and S. Zhang, \emph{When Sets Can and Cannot Have MSTD Subsets}, Journal of Integer Sequences \textbf{21} (2018), Article 18.8.2. 

\bibitem[DKMMW]{DKMMW}
T. Do, A. Kulkarni, S. J. Miller, D. Moon and J. Wellens, \emph{Sums and Differences of Correlated Random Sets}, Journal of Number Theory \textbf{147} (2015), 44--68.

\bibitem[DKMMWW]{DKMMWW}
T. Do, A. Kulkarni, S. J. Miller, D. Moon, J. Wellens and J. Wilcox, \emph{Sets Characterized by Missing Sums and Differences in Dilating Polytopes}, Journal of Number Theory \textbf{157} (2015), 123--153.

\bibitem[He]{He}
P. V. Hegarty, \emph{Some explicit constructions of sets with more sums
than differences}, Acta Arith. \textbf{130} (2007), 61--77.

\bibitem[HM]{HM}
P. V. Hegarty and S. J. Miller, \emph{When almost all sets are difference
dominated}, Random Structures Algorithms \textbf{35} (2009),
118--136.

\bibitem[HLM]{HLM}
A. Hemmady, A. Lott and S. J. Miller, \emph{When almost all sets are difference dominated in $\mathbb{Z}/n\mathbb{Z}$},  Integers \textbf{17} (2017), Paper No. A54, 15 pp.

\bibitem[ILMZ]{ILMZ}
G. Iyer, O. Lazarev, S. J. Miller and L. Zhang, \emph{Generalized More Sums Than Differences Sets}, Journal of Number Theory \textbf{132} (2012), no. 5, 1054--1073.

\bibitem[LMO]{LMO}
O. Lazarev, S. J. Miller and K. O'Bryant, \emph{Distribution of Missing Sums in Sumsets}, Experimental Mathematics \textbf{22} (2013), no. 2, 132--156.

\bibitem[MA]{Ma}
J. Marica, \emph{On a conjecture of Conway}, Canad. Math. Bull. \textbf{12} (1969), 233--234.

\bibitem[MO]{MO}
G. Martin and K. O'Bryant, \emph{Many sets have more sums than differences}, Additive Combinatorics, Providence, RI, 2007, 287--305.

\bibitem[MOS]{MOS}
S. J. Miller, B. Orosz and D. Scheinerman, \emph{Explicit constructions of infinite families of MSTD sets}, Journal of Number Theory \textbf{130} (2010), 1221--1233.

\bibitem[MP]{MP}
S. J. Miller and C. Peterson, \emph{A geometric perspective on the MSTD question}, Discrete and Computational Geometry \textbf{62} (2019), no. 4, 832--855.

\bibitem[MS]{MS}
S. J. Miller and D. Scheinerman, \emph{Explicit constructions of infinite families of MSTD sets}, Additive Number Theory: Festschrift In Honor of the Sixtieth Birthday of Melvyn B. Nathanson (David Chudnovsky and Gregory Chudnovsky, editors), Springer-Verlag, 2010.

\bibitem[MV]{MV}
S. J. Miller and K. Vissuet, \emph{Most Subsets are Balanced in Finite Groups}, Combinatorial and Additive Number Theory, CANT 2011 and 2012 (Melvyn B. Nathanson, editor), Springer Proceedings in Mathematics \& Statistics (2014), 147--157.

\bibitem[MXZ]{MXZ}
S. Miller, V. Xu and X. Zhang,  \emph{MSTD Subsets and Properties of Divots in the Distribution of Missing Sums}, Combinatorial and Additive Number Theory, 05/26/16.

\bibitem[Na1]{Na1}
M. B. Nathanson, \emph{Problems in additive number theory I}, Additive combinatorics, Providence, RI, 2007, 263--270.

\bibitem[Na2]{Na2}
M. B. Nathanson, \emph{Sets with more sums than differences}, Integers \textbf{7} (2007), \#A5.

\bibitem[Ru1]{Ru1}
I. Z. Ruzsa, \emph{On the cardinality of $A + A$ and $A - A$}, Combinatorics Year, North-Holland-Bolyai T$\grave{{\rm a}}$rsulat, Keszthely, 1978, 933--938.

\bibitem[Ru2]{Ru2}
I. Z. Ruzsa, \emph{Sets of sums and differences}, S\'eminaire de
Th\'eorie des Nombres de
Paris, Birkh\"auser, Boston, 1984, 267--273.

\bibitem[Ru3]{Ru3}
I. Z. Ruzsa, \emph{On the number of sums and differences}, Acta Math. Sci. Hungar. \textbf{59} (1992), 439--447.

\bibitem[Zh1]{Zh1}
Y. Zhao, \emph{Constructing MSTD sets using bidirectional ballot
sequences}, J. Number Theory \textbf{130} (2010),
1212--1220.

\bibitem[Zh2]{Zh2}
Y. Zhao, \emph{Sets characterized by missing sums and differences}, J. Number Theory \textbf{131} (2011), 2107--2134.

\end{thebibliography}
\end{document}